\documentclass[12pt,reqno]{amsart}
\usepackage[dvips]{graphics}
\usepackage{amssymb}
\usepackage{dsfont}

\numberwithin{equation}{section}

\newtheorem{thm}{Theorem}[section]
\newtheorem*{thm*}{Theorem}
\newtheorem*{thmmain*}{MAIN THEOREM}
\newtheorem{lem}[thm]{Lemma}
\newtheorem{cor}[thm]{Corollary}
\newtheorem{prop}[thm]{Proposition}
\newtheorem*{prop*}{Proposition}
 
\theoremstyle{definition}

\theoremstyle{remark}
\newtheorem{rem}{Remark}[section]

\newcommand{\tref}[1]{Theorem~\ref{#1}}
\newcommand{\cref}[1]{Corollary~\ref{#1}}
\newcommand{\pref}[1]{Proposition~\ref{#1}}

\newcommand{\lref}[1]{Lemma~\ref{#1}}

\def\dim{\mathop{\text{dim}}}

\begin{document}
\title{Polar foliations of symmetric spaces}

\author{Alexander Lytchak}
\address{A. Lytchak, Universit\"at zu K\"oln, Mathematisches Institut, Weyertal 86-90,
50931 K\"oln, Germany}
\email{alytchak\@@math.uni-koeln.de}



\begin{abstract}
We prove that a polar foliation of codimension at least three in an irreducible compact symmetric space 
is hyperpolar, unless the symmetric space has rank one. For reducible symmetric spaces of compact type, we
 derive decomposition results for polar foliations.
\end{abstract}

\thanks{ The author was supported   by a  Heisenberg grant of the DFG and by the  SFB  878
{\it Groups, Geometry and Actions}}

\maketitle
\renewcommand{\theequation}{\arabic{section}.\arabic{equation}}

\pagenumbering{arabic}

\section{Introduction}
The following is the most important special case of our results:

\begin{thm} \label{irredthm}
Let $M$  be a simply connected,  irreducible, non-negatively curved   symmetric space,  and let $\mathcal F$ 
be a polar foliation on $M$ of codimension at least 3.
Then either all leaves of $\mathcal F$ are points, or $\mathcal F$ is hyperpolar, or the symmetric space  has rank one.
Moreover,  in the last case, $M$ is not the Cayley plane and the foliation lifts via the Hopf fibration to a polar foliation of the round sphere.
\end{thm}

The result confirms a folklore conjecture in the field of polar foliations and actions.  We explain the origin,  ambience and  generalizations of this result below.  But first, we would like to emphasize that the main idea of the proof may be more interesting than the result itself.  Namely, the main step of the proof is 
an application of the famous  theorem of Tits which classifies spherical Tits buildings and shows  that spherical buildings of dimension at least $2$ are homogeneous and of algebraic origin. 
This combinatorial-algebraic theorem, seemingly very far remote 
from the world of Riemannian geometry,  was already used as the main tool in two very important papers \cite{Th2} and \cite{burns}.
Here we provide another application of this classification result along with another theorem of Tits on chamber systems slightly more general than buildings.

 We hope that this approach via the combinatorial  theory of buildings  might be fruitful in the field of polar foliations, providing insights barely attainable by direct differential-geometric means. Our hope is supported by the fact that independently 
 the same idea was applied by F. Fang, K. Grove and G. Thorbergsson in \cite{gro}, to obtain a classification
 of polar actions  on positively curved manifolds.  Moreover,
 this approach gives more weight and interest to  the  exotic 2-dimensional combinatorial objects which has  appeared in the work of Tits on chamber complexes and buildings.  It seems to come as  a surprise in the building community, that such 2-dimensional local
 buildings not coming as quotients of global buildings,  which were discovered in the 80ies
 (\cite{Neu}) and  considered as something extremely bizarre and  exotic, appear naturally as problems and examples in Riemannian geometry.
  We refer the reader to \cite{KrL} for more on these combinatorial objects, finish the advertisement part and come to the introduction.

 A \emph{polar foliation} $\mathcal F$ of a complete $m$-dimensional Riemannian manifold $M$ is a singular Riemannian foliation  (cf. Section \ref{2}) with 
regular leaves of dimension $(m-k)$, such that  each point $x\in M$ is contained in a complete, totally geodesic, immersed submanifold $\Sigma$ of dimension $k$, by definition, the \emph{codimension of the foliation},  that intersects all leaves of $\mathcal F$ orthogonally.   Such a submanifold is called a  \emph{section}
 of $\mathcal F$.    The polar foliation $\mathcal F$
is called \emph{hyperpolar} if one and hence all sections are flat.    If the foliation is given  by the orbit foliation
 of an isometric  action it is called \emph{homogeneous} and the action is called  a \emph{polar action}.

In space forms, the investigation of polar foliations of codimension one has been initiated by Segre and Cartan
and in higher codimensions by Terng (\cite{Te}) under the name of \emph{isoparametric foliations}.   We refer to the excellent surveys  \cite{Th5}, \cite{Th6} and the huge 
list of references therein.  It turns out that polar foliations in Euclidean spaces come from polar foliations on spheres.
Polar foliations of codimension at least two   in round  spheres have  been shown by Thorbergsson  to be homogeneous (if they are "irreducible and full")
and  related to {\it non-positively curved} symmetric spaces and their buildings at infinity (\cite{Th2}).   
On the other hand, in codimension one, there are series of inhomogeneous examples and the classification is still not complete, despite great recent progress in the area (\cite{FKM}, \cite{Stolz}, \cite{Chi}, \cite{immer2}, and \cite{Th6} 
for more references).

The investigation of polar foliations in (from now on always) non-negatively curved symmetric spaces $M$
has been initiated in \cite{TTh1}. It has been shown that, using a Riemannian submersion $\mathcal H \to M$
from a Hilbert space of paths to $M$, one can lift hyperpolar (!) foliations from $M$ to $\mathcal H$.
From this observation one could "understand" all  hyperpolar ("full, irreducible") foliations of codimension at least 2,
by showing that they are homogeneous (\cite{He-Li}, \cite{Ewert}, \cite{Christ}). In irreducible symmetric spaces, such hyperpolar actions have been classified in \cite{K}.

On the other hand, in symmetric spaces of rank one, there are lots of polar foliations (cf. \cite{Pth1}, \cite{Dominguez})
which are never hyperpolar if the codimension is at least two. Motivated by the known examples and confirmed by   the partial classification of polar {\it actions} on irreducible symmetric spaces obtained in  \cite{Kol1}, it has been conjectured that
polar foliations on irreducible symmetric spaces of higher rank are hyperpolar.  Our \tref{irredthm} confirms this conjecture
if the codimension is not equal to two.

In our approach, the irreducibility of $M$ does not play a major role. More important is the irreducibility of the sections, more
precisely of the quotient spaces $M/\mathcal F$.    Without the assumption of irreducibility we prove:

\begin{thm} \label{mainthm}
Let $\mathcal F$ be a polar foliation on  a simply connected non-negatively curved symmetric space $M$.
Then we have a splitting $M=M_{-1} \times M_0\times  M_1\times ... \times M_l$, such that $\mathcal F$
is a direct product of polar foliations $\mathcal F_i$ on $M_i$.  
The foliation $\mathcal F_{-1}$ on $M_{-1}$ is given by the fibers of the projection of $M_{-1}$ onto a direct factor
of $M_{-1}$.
The foliation $\mathcal F_0$ is hyperpolar.
For $i\geq 1$,  the sections of  the foliation $\mathcal F_i$ on $M_i$ have constant positive sectional curvature.
For $i\geq 1$, if the codimension of $\mathcal F_i$ on $M_i$ is at least $3$ then $M_i$  is irreducible and of  rank one; moreover, in this case, the  foliation 
$\mathcal F_i$ lifts to a polar foliation of the round sphere. 
\end{thm}

 In the special case of \emph{polar action} it is possible to understand the arising additional difficulties in cohomogeneity two. Based on this work, in \cite{KrL} (and, previously, in \cite{KolL}, in the irreducible case),  it is proved that the  additional assumption on the codimension being at least $3$ is redundant.  Thus any  polar action on a non-negatively curved symmetric space is (up to orbit-equivalence) a direct product of hyperpolar actions and of 
 polar actions on spaces of rank $1$.

 The method of proving our main result is inspired by the proof of the  homogeneity result of polar foliation in Euclidean spaces
due to Thorbergsson \cite{Th2}.  We reduce the statement to the case in which the sections have constant curvature $1$. We investigate the horizontal object of our foliation, that is a length metric space
defined by measuring the lengths of broken horizontal geodesics with respect to the foliation.  We use Wilking's results about 
the dual foliation to see that (in the irreducible case) this new metric space is connected.   Since the local structure of this metric space is given by polar foliations on the Euclidean space, this new metric space is locally isometric to some spherical building (possibly
up to one special  case that can be handled directly). Now we
use a theorem of \cite{charney}, stating that, if the codimension $k$ of the foliation (i.e., the dimension of our horizontal object) is at least $3$, this horizontal space is covered by a spherical building.    Moreover,  we use our coarser manifold  topology, to find a coarser compact topology on our building. 
If this building is reducible, one can use direct methods to detect the structure of our symmetric space. In the "main" irreducible case,    we use the theorems of Burns-Spatzier and Tits (\cite{burns}, \cite{Tits})  saying that our building is the building of a simple non-compact real Lie group. In particular, its coarser topology is that of a sphere.
 Then  our manifold turns out to be  the base of a principal  fibration of a sphere. Therefore  it  is homeomorphic  to a projective space. We conclude that our symmetric space  has rank $1$.

Finally, we would like to mention that the case of cohomogeneity $2$ is  different   not only for technical reasons. The main point is that the universal covering of our horizontal space need not be a building (i.e. the  local-global result
from \cite{charney} may fail).  We are aware of only one example in which this problem arises, namely for the polar action
of  $\mathrm {SU} (3) \cdot \mathrm {SU} (3)$ on  the Cayley projective plane $\mathrm {Ca} \mathbb {P}^2$.
In \cite{KrL}, it is shown that this is the only such example 
in the case of polar \emph{actions}, i.e., in the case, when the 
horizontal simplicial object is homogeneous.  Unfortunately, in the general case,
nothing is known about the combinatorial structure of arising objects.

  In  Section \ref{2} we shortly recollect all notions and results needed later in the proof. In Section \ref{3} and 
Section  \ref{4}  we study dual foliations and derive the product decomposition of \tref{mainthm}, reducing 
\tref{mainthm} and \tref{irredthm} to the case where sections have constant positive curvature and the dual foliation has only one leaf.  In Section \ref{6} we introduce our main tool: the horizontal singular metric $d^{hor}$ on our manifold 
$M$ and study its basic properties. It turns out that there are two essentially different cases to be investigated,
depending on whether the spherical Coxeter group in question is reducible or not.  In Section \ref{7} we study the reducible case and apply some basic results of \cite{Nagano}  about special totally geodesic subspaces, called \emph{polars} and
\emph{meridians}, to prove that   our symmetric space $M$ has rank $1$.  In Section \ref{8}, together with 
Section \ref{6}, the heart of the paper, we use the fact the universal covering of our singular metric space $(M, d^{hor})$
is a spherical building.  We define a coarser topology on this space and use the main theorem of \cite{burns}
to prove that this coarser topology is the topology of a sphere.  Then we deduce   that $M$ has rank $1$.
In the final section, we use a simple argument inspired by  \cite{Pth1}, to  describe 
polar foliations on symmetric spaces of rank 1, thus finishing the proof of our main theorems.

\section{Preliminaries} \label{2}
\subsection{Foliations}
We refer to \cite{Wilk}, \cite{LT}, \cite{Lyt} for more on singular Riemannian foliations. Here we just recall the basic notions.
Let $M$ be a  Riemannian manifold.  A singular Riemannian foliation $\mathcal F$ on $M$ is a decomposition of 
$M$ into smooth, injectively immersed submanfolds  $L(x)$, called \emph{leaves}, such that it is a singular foliation and such that
any geodesic starting orthogonally to a leaf remains orthogonal to all leaves it intersects. Such a geodesic is called a
\emph{horizontal geodesic}.  For all $x\in M$,  we denote by $H_x$ the orthogonal complement to the 
tangent space $T_x (L(x))$, and call it the \emph{horizontal space at $x$}.  A leaf  and all of its points are called 
\emph{regular} if it has  maximal dimension.  On the set of regular points, the foliation is locally given by a Riemannian submersion.  The dimension of the regular leaves is called the dimension of the foliation, and their codimension in $M$
is called the codimension of the foliation.

 The foliation is called  \emph{polar} if through any point $x\in M$ one  finds a totally geodesic submanifold whose dimension equals  the codimension of $\mathcal F$ and which  intersects all leaves orthogonally.  This happens if and only if 
the horizontal distribution in the regular part is integrable.   If $M$ is complete, then the totally geodesic submanifolds 
can be chosen to be complete. They are called section of the polar foliation $\mathcal F$. We refer to \cite{Alex}, \cite{Alext}, \cite{Lyt}  for more on polar foliations.

If the foliation $\mathcal F$  is polar and $M$ is simply connected then all leaves are closed. 
 The quotient space (the space of all leaves)
will be denoted by $\Delta$. It comes along with the canonical projection $p:M\to \Delta$ which is a \emph{submetry}.
The quotient $\Delta$ is a \emph{good Riemannian Coxeter orbifold} (\emph{reflectofold}, in terms of  \cite{Davis}).    Moreover, the restriction  $p:\Sigma  \to \Delta$ to
any section $\Sigma$ is a Riemannian branched covering.  Thus $\Delta $ is isometric to $\tilde \Sigma  /\Gamma$, where
$\tilde \Sigma$ is the universal covering of $\Sigma$ and $\Gamma$ is a \emph{reflection group}, i.e., a discrete group of isometries of $\tilde {\Sigma }$ generated 
by reflections at totally geodesic hypersurfaces.

For any point $x\in M$, the singular Riemannian foliation defines an infinitesimal singular Riemannian
foliation $T_x \mathcal F$ on $T_x M$,  that factors as a projection of $T_x M$ to $H_x$ and the restriction
of $T_x \mathcal F $ on $H_x$.   If $\mathcal F$ is polar then $T_x \mathcal F$ is polar and sections of $\mathcal F$ through $x$
are in one-to-one correspondence with sections of $T_x \mathcal F$ through the origin.  Any horizontal geodesic is contained in a section of $\mathcal F$.  Moreover, either the foliation is regular or  there are two sections $\Sigma _{1,2}$ whose intersection $\Sigma _1 \cap \Sigma _2$ is a hypersurface in both sections $\Sigma _{1,2}$.

\subsection{Dual foliation}
The dual foliation $\mathcal F ^ {\#}$ of a singular Riemannian foliation  $\mathcal F$ is defined by letting the leaf $L^ {\#} (x)$  be the set of all points
in $M$ that can be connected with $x$ by a broken horizontal geodesic.  In \cite{Wilk} it is shown  that
$\mathcal F ^{\#}$ is indeed a singular foliation.  The  following important results has been shown in \cite{Wilk} 
(we use slightly weaker formulations, suitable for our aims):

\begin{prop} \label{dual}
Let $M$ be a complete non-negatively curved manifold with a singular Riemannian foliation $\mathcal F$.
Let  $\gamma$ be an $\mathcal F$-horizontal geodesic starting at a point $x\in M$.  Let $W(t) := \nu _{\gamma (t)} L ^{\#} (x)$ denote 
the normal space to the dual leaf $L ^{\# } (x) =L^{\# } (\gamma (t))$  at the point $\gamma (t)$.   Then $W(t)$ is parallel along 
$\gamma$.  Moreover, for all $w\in W(t)$, the sectional curvature $sec (w \wedge \gamma ' (t))$  of the plane spanned by
$w$ and $\gamma ' (t)$ is $0$.  
\end{prop}

\begin{prop} \label{complete}
Under the assumptions of \pref{dual}, if  all dual leaves are complete in their induced  metric  then  the dual foliation
is a singular Riemannian foliation. 
\end{prop}

We provide an easy application of these results:

\begin{lem} \label{easyprod}
Let $M$ be a simply connected, complete, non-negatively curved manifold.  If $\mathcal F$ is a polar {\it regular} foliation
of $M$ then $M$ splits isometrically as a product $M=M_1 \times M_2$ and $\mathcal F$ is given by the projection
$p_1 :M\to M_1$.
\end{lem}

\begin{proof}
By definition, the leaves of the dual foliation $\mathcal F^\#$ are exactly the sections of $\mathcal F$. In particular,
they are complete. Due to \pref{complete}, $\mathcal F^\#$ is a Riemannian foliation as well. Moreover,  the horizontal distribution of $\mathcal F^\#$ coincides with $\mathcal F$, hence it is integrable. Thus the leaves of $\mathcal F$ are the sections of $\mathcal F^\#$. Thus they are totally geodesic. A polar foliation with  totally geodesic leaves is locally given by a projection onto a section, which is locally a direct factor of $M$.  Since $M$ is simply connected, we get a global decomposition $M=M/\mathcal F  \times M /\mathcal F^\#$. 
\end{proof}

\subsection{Spherical Coxeter groups}
A \emph{spherical Coxeter group} is a reflection group $\Gamma$ on a round sphere $\mathbb S^k$.
We will call it reducible if the corresponding action on $\mathbb R^{k+1}$  is reducible.
There is a unique decomposition $\Gamma =\Gamma _1 \times \Gamma_ 2 \times ...\times \Gamma_l$
 and a $\Gamma$-invariant orthogonal decomposition $\mathbb R^{k+1} = V_0 \oplus V_1 ...\oplus V_l$,
such that $\Gamma _i, i=1,...,l$ acts as an irreducibe  reflection group on $V_i$ and trivial on all $V_j, j\neq i$.  

The quotient $\Delta = \mathbb S^k /\Gamma$ is the spherical join $\Delta =\Delta _0 \ast \Delta _1 ...\ast \Delta _l$
 of the round sphere $\Delta _0$ and irreducible Coxeter
simplices $\Delta _i = S_i /\Gamma _i$, where $S_i$ is the unit sphere of $V_i$.

The group $\Gamma$ is called \emph{crystallographic}, if  all dihedral angles of the spherical polyhedron 
$\Delta$ are given by $\pi /m$,  where $m$ can only take the values $1,2,3,4,6$.   If  none of the  direct factors
 $\Gamma_ i$ is one-dimensional, then  none of the dihedral angles of $\Delta$ is equal to $\pi/6$.

Assume now that $\Delta = \mathbb S^k /\Gamma$ is the quotient $\Delta = M/\mathcal F$ of a polar foliation $\mathcal F$
on a simply connected manifold $M$. Take a point $y$ in a  face of $\Delta$ of codimension $2$ in $\Delta$.
Take a point $x$ in the  leaf over $y$. Then the tangent space $T_y \Delta$ is the quotient space of
the polar foliation $T_x \mathcal F$ on $T_x M$ (cf. \cite{Lyt}).  The famous theorem of M\"unzner  (\cite{Mu1}, \cite{Mu2})
implies that the dihedral angle at $y$ can be only given by $\pi/m$, with $m=1,2,3,4,6$.  We deduce that
$\Gamma$ is crystallographic.

\subsection{Spherical buildings}
We define buildings as metric spaces in contrast to their original simplicial definition of Tits. 
We refer to \cite{charney} and \cite{kleiner} for more on buildings considered from our point of view. 
Let $\Gamma$ be a spherical Coxeter group acting on $\mathbb S^n$. A \emph{spherical building} of type $\Gamma$ is a metric space 
$X$ with a  set of isometric embeddings $\phi: \mathbb S^n \to X$, called \emph{apartments}, such that the following two conditions hold true:  
Any pair of points of $X$ is contained in some  apartment and the transition maps between different apartments are given by restrictions of elements of $\Gamma$. 

Consider the natural decomposition of $\mathbb {S}^n$ into  polyhedra isometric to  $\mathbb S^n/\Gamma$. This polyhedral structure is preserved by 
$\Gamma$,  hence we obtain a natural polyhedral structure on $X$.  The building $X$ is called thick if all walls of codimension $1$ bound at least $3$ polyhedra. 

A spherical join of spherical buildings is a spherical building, in particular so is the suspension of a spherical building 
(cf. \cite{Bridson}, for spherical joins and suspensions).
A spherical building $X$ is called irreducible if it is indecomposable as a spherical join. For a thick building of dimension at least 
$1$ this is equivalent to the irreducibility of the Coxeter group $\Gamma$.

\subsection{Obtaining new foliations}
Let again $p:M\to \Delta$ be the projection whose fibers are leaves of a polar foliation $\mathcal F$ on $M$.
Write again $\Delta = N/\Gamma$, where $N$ is the universal covering of any section.
 Assume that there is a $\Gamma$-invariant polar foliation 
$\mathcal G$ on $N$. Then $\Gamma$ acts on the quotient   orbifold  $ N/\mathcal G $ by isometries.  Assume 
that this action has closed (i.e. discrete) orbits and let $\Delta '$  be the quotient orbifold 
$\Delta ' =  (N/\mathcal G) /\Gamma$.  The projection $N\to \Delta '$ factors by definition through $\Delta$.

Then  the composition $p' =q\circ p  :M\to \Delta '$  
is the quotient map of a new polar foliation $\mathcal F '$ on $M$.

Namely, $p$ is a submetry (i.e., its fibers are equidistant) and so is $q$, hence
the fibers of $p'$ are equidistant as well.  Around the preimage of a regular point of $\Delta'$, $p'$ is the composition of two Riemannian submersions with  sections, hence it is itself a Riemannian submersion with sections.  It  only remains  to prove that $\mathcal F'$ is a singular foliation. This can be  done directly.  A slightly more elegant and sophisticated proof
is obtained as follows.  It is a direct consequence of our construction and the main definition of  \cite{Toe},
that the regular  fiber has  \emph{parallel focal structure}. The main result of \cite{Toe} now implies that $\mathcal F '$
is a singular Riemannian foliation.

By construction, each dual leaf of $\mathcal F'$ is contained in a dual leaf of $\mathcal F$.   On the other hand,
if the polar foliation $\mathcal G$ on $N$ has only one dual leaf, then the dual leaves of $\mathcal F$ and of
$\mathcal F'$ coincide.

We are going to use this construction only in two simple cases.
First assume that $\Delta $ is a direct metric product $\Delta = \Delta '\times \Delta ''$.
Then the composition $p'$ of $p:M\to \Delta$ and the projection $q:\Delta \to \Delta '$ defines a polar foliation 
on $M$.

We will consider only one other case.  Assume that $\Delta $ is given as the quotient
$\mathbb S^k/\Gamma$, where $k\geq 2$ and $\Gamma$ is a spherical  Coxeter group. Assume that $\Gamma$ is reducible.
Consider the $\Gamma$-invariant orthogonal decomposition $\mathbb R^{k+1} =V_1 \oplus V_2$.
Then $\Delta $ is a spherical  join  $\Delta = \Delta _1 \ast \Delta _2$. 
Collapsing $\Delta _i$ to points, we obtain a projection $\Delta \to [0,\pi /2]$, which corresponds to the reducible, polar, codimension one foliation on $\mathbb S^k$  which is given by the distance function $p': \mathbb S^k \to [0,\pi/2]$  to the sphere 
$\mathbb S^k \cap V_1$.

Note, that any non-trivial singular Riemannian foliation on the round sphere has only one dual leaf, due to \cite{Wilk}.
Collecting the previous observations we arrive at:

\begin{lem} \label{late}
Let $\mathcal F$ be a polar foliation on a complete Riemannian manifold $M$. Assume that the quotient 
$\Delta$ is isometric to $\mathbb S^k/\Gamma$, with  a reducible Coxeter group $\Gamma$.
Then there is a coarser polar foliation $\mathcal F'$ on $M$, which has the same dual leaves as $\mathcal F$ and 
whose quotient space $\Delta ' $ is the interval $[0,\pi /2]$.
\end{lem}

\section{Dual foliations on symmetric spaces} \label{3}
We will use a general observation about dual leaves in symmetric spaces.
  Marco Radeschi has pointed out  
that a variant of the two subsequent results appeared in  \cite{Tapp}.  

\begin{prop}   \label{connected}
Let $M$ be a non-negatively curved symmetric space. Let $\mathcal F$ be a singular Riemannian foliation on $M$ and let
$\mathcal F ^ \#$  be the dual foliation. Then any leaf  $L ^ \#$ of the dual foliation is contained in a 
totally geodesic submanifold $Z$ of the same dimension as $ L ^\#$. Moreover, $Z$ is a direct factor of $M$.
In particular, if the dual leaf $ L ^\#$ is complete, it is a direct factor of $M$. 
\end{prop}

\begin{proof}
Take a point $x\in  L^\#$.  Let $W_x$ denote  the normal space $W_x=\nu _x (L ^\# )$ to the dual leaf.
We let $W_x '$ be the set of all vectors $w'$ in $T_x M$ such that, for all $w\in  W_x$, the sectional curvature 
$sec (w\wedge w' )$ is $0$.   Identifying $x$ with the origin of the symmetric space $M=G/K$ and writing
$\mathfrak g =\mathfrak k \oplus \mathfrak p$, with the usual identification of $\mathfrak p$ and $T_x M$,
we have 
$$W _x  ' := \{ w \in T_x M |  [W_x ,w] =0 \}$$

Due to the Jacobi identity, $[[W'_x ,W' _x],W'_x ] \subset W' _x$. Thus, by definition, $W' _x$ is a \emph{Lie triple system}.
The subspace $W'_x \cap W_x$ commutes with $W'_x$.  Hence the orthogonal complement $W'' _x$ of 
$(W_x \cap W' _x)$ in $W' _x$  is a Lie triple system as well.   Exponentiating the Lie triple system $W'' _x$,
we obtain a totally geodesic submanifold $Z= \exp (W _x'' )$. By definition, $W_x '' \cap W_x = \{ 0 \}$. Hence $\dim (Z)\leq \dim (M) - \dim (W_x) = \dim (L^\# (x))$.

We are going to prove that $Z$ contains $L^\# (x)$. Take an  $\mathcal F$-horizontal broken geodesic $\gamma$ that starts in $x= x_1$ and consists of a finite  concatenation of   
$\mathcal F$-horizontal geodesics $\gamma_i$ connecting $x_i$ and $x_{i+1}$. 
Due to \pref{dual}, the starting direction of $\gamma _i$ is contained in $W'' _{x_i}$.
Moreover, the 
 parallel translation along $\gamma _i$ sends $W_{x  _i} $ to $W_{x_{i+1}}$. 

Since the curvature tensor is invariant under parallel translation in the  symmetric space $M$, the parallel translation along $\gamma _i$ sends $W_{x_i} ''$ to  $W_{x_{i+1}}''$.  By induction on the number of concatenations, we deduce that $\gamma$ is contained in $Z$.
Since any point of $L_x  ^\#$ can be reached from $x$ by a broken $\mathcal F$-horizontal geodesic, we deduce that
$ L _x ^ \# $  is contained in $Z$.

   Thus we must have $\dim (Z)= \dim (L^\# (x))$. Then $W_x$ and $W''_x$ are complementary commuting subspaces of $\mathfrak p$.  Therefore,  $W_x$ is a Lie triple system as well, and $M$ splits as the product of $Z$ and its orthogonal complement.
\end{proof}

In particular, we deduce:

\begin{cor}  \label{cordual}
If $\mathcal F$ is a singular Riemannian foliation on a compact irreducible symmetric space then the dual foliation has only one leaf, unless  $\mathcal F$ has only one leaf.
\end{cor}

Another  consequence, we will use is:

\begin{cor} \label{factor}
Let $\mathcal F$ be a singular Riemannian foliation on a simply connected symmetric space $M$. If the dual leaves of $\mathcal F$ are
 complete  then $M$ splits as $M=M_1 \times M_2$ such that the dual leaves of $\mathcal F$ are exactly the $M_2$-factors,
i.e., all dual leaves have the form $\{ x_1  \} \times M_2$.
\end{cor}

\begin{proof}
Due to \pref{complete}, the dual foliation is a singular Riemannian foliation.  Due to \pref{connected},
all leaves must be factors of $M$. Since these factors are equidistant they must be  $M_2$-factors of the
same   product decomposition $M=M_1 \times M_2$.
\end{proof}

\section{Product decomposition}  \label{4}
Here and  in the sequel,  let  $M$ be a simply connected non-negatively curved symmetric space and let 
$\mathcal F$ be a polar foliation on $M$.  
\subsection{Decomposition of the factor}
We recall that our foliation has closed leaves and  that the quotient space  $\Delta =M/\mathcal F$  is a  Coxeter orbifold.
Moreover,  $\Delta$ is a discrete  quotient of a section, the last being  a totally geodesic submanifold of $M$,
hence a symmetric space itself.  Thus $\Delta $ is given as  $\Delta =N/\Gamma$ with a  symmetric non-negatively curved simply connected manifold $N$ (the universal covering of a section $\Sigma$), on which $\Gamma$, the orbifold fundamental group of $\Delta$, acts as a reflection group.

Let $N=N_0 \times N_1 \times ....\times N_l$ be the direct product  decomposition, where $N_0$ is the Euclidean space and 
where $N_i$ are irreducible of dimension at least $2$.  Any reflection (always at a wall of codimension 1!) on $N$   respects this product decomposition. Hence it induces a reflection on one factor
and identity on all  other factors.  Therefore,  $\Gamma$ is a direct product
 $\Gamma =\Gamma _0  \times \Gamma _1 \times ... \times \Gamma _l$, where $\Gamma _i$ is the subgroup of $\Gamma$
generated by all reflections fixing all factors $N_j,j\neq i$.  Moreover, the quotient $\Delta$ splits isometrically as
the direct product $\Delta = \Delta _0 \times ...\times \Delta _l$, with $\Delta _i =N_i /\Gamma _i$. 

Finally, the only simply connected irreducible symmetric spaces of compact type  which admit a totally geodesic hypersurface  are round spheres. Thus, for all $i\geq 1$,  either $N_i$ is a round sphere or $\Gamma _i$ is trivial. Therefore,  in the above product decomposition all $\Delta _i, i\geq 1$
either have constant positive curvature or they coincide with the   Riemannian manifolds $N_i$ 
(this fact  has been observed in \cite{Kol1}).

\subsection{Decomposition of the space}
We call a polar foliation $\mathcal F$ on a symmetric space $M$   \emph{decomposable}  if $M$ can be decomposed
non-trivially as $M=M_1 \times M_2$ such that $\mathcal F$ splits as $\mathcal F=\mathcal F_1 \times \mathcal F_2$,
a product  of polar foliations on the factors.  Otherwise we call $\mathcal F$  \emph{indecomposable}.

 The proof of the following observation is postponed to Section \ref{6}:

\begin{lem}  \label{postpone}
Assume that the sections of $\mathcal F$ have constant positive curvature.  Then the dual leaves are compact. In particular,
they are factors of $M$.
\end{lem}

Now we can prove:

\begin{prop}
 Let $\mathcal F$ be indecomposable.  Then either $\mathcal F$ is trivial, or  hyperpolar, or $\Delta $ has constant positive curvature and the dual foliation $\mathcal F ^\#$ has only one leaf.
\end{prop}

\begin{proof}
 Assume that $\Delta$ is non-trivially decomposed as $\Delta _1 \times \Delta _2$, with $\Delta_1$ either a manifold or of constant positive curvature.       Consider the induced submetry $p_1 :M\to \Delta _1$ that is given by a polar foliation 
$\mathcal F_1 '$.    Due to the preceding lemmas (\lref{easyprod}, \lref{postpone}, \cref{factor}), the  leaves of the dual foliation of $\mathcal F_1 '$ are $M_1$-factors
in a decomposition $M=M_1 \times M_2$.  

Any $\mathcal F_1 '$-horizontal geodesic  is mapped by the projection to $\Delta _1 \times \Delta _2$ into a $\Delta _1$ factor, hence by the projection $p_2 :M\to \Delta _2$ to a point. This shows that any dual leaf of $\mathcal F_1 '$ is contained  in a leaf  of the foliation $\mathcal F_2 '$ defined by the projection $p_2 :M\to \Delta _2$.   Thus the foliation
$\mathcal F_2'$  is coarser than the foliation defined by the $M_1$-factors. Hence, $p_2$ factors through  the projection 
$q_2 :M\to M_2$.   

Taking $\mathcal F_1$ to be the restriction of $\mathcal F$ to $M_1$ (any $M_1$-factor) and $\mathcal F_2$ the restriction of $\mathcal F$ to $M_2$ (any $M_2$-factor) we get $\mathcal F=\mathcal F_1 \times \mathcal F_2$. 
This contradicts to the assumption  that  $\mathcal F$ is indecomposable.

If a decomposition of $\Delta$ as above does not exist, then $\Delta$ must be either  flat, or  a manifold, or of  constant positive curvature. If it is flat then  the foliation is hyperpolar. If $\Delta$ is a manifold  then $\mathcal F$ must be given by a projection to a factor. Since $\mathcal F$ is indecomposable,  this  factor and, therefore, $\mathcal F$ must be trivial.
In the remaining case, $\Delta$ must have constant positive curvature. Then  by \lref{postpone} and \cref{factor},
the dual foliation has only one leaf.  
\end{proof}

\begin{rem}
Note, that the hyperpolar factor may be decomposed further until the quotient $\Delta$ is irreducible (\cite{Ewert}).
\end{rem}

Given a polar foliation $\mathcal F$ on $M$, we  now decompose it in indecomposable pieces. Taking trivial pieces
together we obtain a foliation given by a projection to a direct factor. Collecting hyperpolar pieces together we get a hyperpolar foliation.
Thus we arrive at:

\begin{prop}  \label{decompose}
Let $\mathcal F$ be  a polar foliation on a  non-negatively curved simply connected symmetric space $M$. 
Then we have a splitting $M=M_{-1} \times M_0\times  M_1\times ... \times M_l$, such that $\mathcal F$
is a direct product of polar foliations $\mathcal F_i$ on $M_i$.  
The foliation $\mathcal F_{-1}$ on $M_{-1}$ is given by the fibers of the projection of $M_{-1}$ onto a direct factor
of $M_{-1}$.
The foliation $\mathcal F_0$ is hyperpolar.
For $i\geq 1$,  the sections of  the foliation $\mathcal F_i$ on $M_i$ have constant positive sectional curvature; moreover, for $i\geq 1$, there is only one dual leaf of $\mathcal F_i$.
\end{prop}

\section{New setting}  \label{5}
Due to \pref{decompose},  in order to prove \tref{mainthm} and \tref{irredthm}, we only need to study the case
in which the sections of $\mathcal F$ have constant positive curvature.

From now on, we will assume the sections of  $\mathcal F$  to have  constant positive  curvature. Hence the sections are either spheres or projective spaces.  
 We normalize  the space
such that the sections and the  quotient  have constant curvature $1$.
 Thus for any horizontal vector $v$ the geodesic in direction $v$ is closed
of period $\pi$ or of period $2\pi$. Since in a symmetric space, for a continuous variation of closed geodesics the period 
of the geodesics cannot change, we deduce that all horizontal geodesics have the same period.  This period  is equal to
$2\pi$ if all sections are spheres,  and it is equal to $\pi$ if all sections are projective spaces.

The quotient $\Delta$ is equal to $\Delta =\mathbb S^k /\Gamma$ for a spherical   Coxeter group $\Gamma$, that
must be crystallographic.  By assumption, $k\geq 2$.

\section{Horizontal metric}  \label{6}
\subsection{Definition}
We now define a new metric  $d^{hor}$ on our manifold $M$ by declaring $d^{hor} (x,y)$ to be the infimum
over all lengths of broken horizontal geodesics that connect $x$ and $y$. By definition $d^{hor } \geq d$. The 
dual leaf $L^\# (x)$ is exactly the set of points that have a finite distance to the point $x$.
We denote by $X$ the set $M$ with the horizontal metric $d^{hor}$.

By construction,  the identity $i:X\to M$ is  $1$-Lipschitz and the projection $p: X\to \Delta$ is still a submetry. 
Since any horizontal geodesic is contained in a section, we see that the metric space  $X$ is defined  by   gluing together  
spherical polyhedra  (each one  isometric to the quotient $\Delta$).  A pair of polyhedra  may be glued only along 
some union of faces.   By definition, $X$ is a length space and since it is a polyhedral complex with only one type of 
polyhedra, it is a geodesic space, i.e., any pair of points at a finite distance are connected by a geodesic with respect to $d^{hor}$ (cf. \cite{Bridson}, p.105).

 Given a point $x\in X$, a small ball $U_x$ around $x$  in $X$ is given by  image  of a small 
 ball in the horizontal space $H_x$  under the exponential map. (Note, however,   that the exponential map, considered as a map from $H_x $ to $X$ is not continuous).
  Consider the induced infinitesimal polar foliation $\mathcal F_x$ on the  Euclidean space $H_x$.  The sections of $\mathcal F$ through $x$ are in one-to-one correspondence with the sections of $\mathcal F_x$. Hence a small neighborhood
of $x$ in $X$ is isometric to a small ball in  the spherical suspension over  the "horizontal metric space"  $Y = (\mathbb S^r, d^{hor})$
 that is defined by the polar  foliation $\mathcal F_x$ on the unit sphere $\mathbb S^r= H_x ^1$ in $H_x$.

Thus, $X$ is a $k$-dimensional locally spherical space in the sense of \cite{charney}.  Moreover,  the space of directions $S_x X$ at each point $x\in X$ is isometric to the horizontal space defined by the  infinitesimal polar foliation
on the unit sphere $H^1 _x$. 

\subsection{Classical case and the irreducible case}
 We are going to use the following result due to Immervoll and Thorbergsson (\cite{immer}, \cite{Th2}): 

\begin{prop} \label{sphbuild}
Let $\mathcal G$ be a polar foliation on the round sphere $\mathbb S^r$. Let $C$ be the quotient $C = \mathbb S^r /\mathcal G$.
If the Coxeter polyhedron $C$ does not have dihedral angles equal to $\pi/6$ then the horizontal space
$Y=(\mathbb S^r,d^{hor})$ defined by the foliation is a spherical building.
\end{prop}

\begin{rem}
The conclusion of the previous lemma is true without any assumptions on the angles, if the foliation comes from a group action, in which case \pref{mainred} below  is a direct  consequence of \pref{sphbuild}. 
\end{rem}

Let  again $M$ be a symmetric space with our polar foliation $\mathcal F$ and quotient $\Delta$ of dimension $k\geq 2$.
We say that $\Delta$ is \emph{irreducible}, if  the corresponding spherical  Coxeter group $\Gamma$ is irreducible. Otherwise, we say that $\Delta$ is \emph{reducible} and find  a non-trivial  decomposition 
$\Delta =\Delta _1 \ast \Delta _2$ of $\Delta$ as a spherical join.

 If $\Delta$ is irreducible, then $\Delta$ does not have faces meeting at the  dihedral angle $\pi/6$.  Therefore we conclude:

\begin{lem}
If $\Delta$ is irreducible  then for any point  $x$ in the  horizontal space $X$, a small neighborhood
 of $x$ is isometric to an open subset of a  spherical building (possibly depending  on the point).
\end{lem}

\begin{rem}
Just to avoid confusion, we remark that in our convention the suspension over a building is again a building of one dimension larger. 
\end{rem}

Applying \cite{charney} we deduce:

\begin{cor} \label{mainirred}
Assume that $\Delta$ is  irreducible.  Let $X'$ be any dual leaf of $\mathcal F$ with the metric induced from $X$. Then $X'$ has diameter  at most $\pi$. If $k=\dim (\Delta ) \geq 3$ then the universal covering of $X'$ is a spherical building. 
\end{cor}

\subsection{Reducible case}  
We are going to prove that in the reducible case, our manifold $M$  possesse two submanifolds that behave like 
a projective subspace and its cut locus in a projective space.

\begin{prop} \label{mainred}
Let $\Delta$ be reducible $\Delta = \Delta _1 \ast \Delta _2$. Let $A_i = p^{-1} (\Delta _i)$ and let $\mathcal F'$
be the polar foliation given by the submetry $p':M\to [0,\pi /2]$, with $p'(x)=d(A_1,x)$.
 Then  for any pair of points $x_i \in A_i$ that are contained
in the same dual leaf $ L^\#$ of $\mathcal F$, we have $d(x_1,x_2) =\pi /2$.
\end{prop}

\begin{proof}
 Since $k=\dim (\Delta ) \geq 2$, at least one $\Delta _i$ is not a point. Without loss of generality, let 
$\Delta _1$ have positive dimension.   Due to \lref{late},
the  dual leaves of $\mathcal F$ and $\mathcal F'$ coincide.  The horizontal metric $d' _{hor}$ induced by 
$\mathcal F'$ on $M$   is the metric on a graph with each edge being a horizontal geodesic from $A_1$ to $A_2$ of length $\pi/2$.

Thus, if the claim of the proposition is wrong, we find $x_1 \in A_1$ and $x_2 \in A_2$ such that there is a shortest geodesic
$\gamma$ with respect to $d' _{hor}$ from $x_1$ to $x_2$ that has length $3\pi /2$.  Let $x_+$ be $\gamma (\pi /2) \in A_2$  and let $x_-$ be $\gamma (\pi ) \in A_1$.

 Consider the polar foliation $T_{x _+} \mathcal  F$ on the Euclidean space $H_{x_+}$.   The quotient is given by the tangent cone to $\Delta$ at a point of $\Delta _2$. Hence it splits as a product of the tangent space to 
$ \Delta _2$ (that may be trivial)
and the orthogonal complement $Q$.  This implies a corresponding splitting of $H_{x_+}$, into a part tangent to 
$A_2$ and the part $H'$ of all $\mathcal F'$-horizontal vectors.    Moreover, we have
$H' /T_{x_+}  \mathcal F = Q$.  
Since the restriction of $\mathcal F'$ to the unit sphere of $H'$ is non-trivial,  it has only one dual leaf.
Thus we find a broken horizontal  geodesic in this sphere that connects the incoming and outgoing direction of $\gamma$ at $x_+$.  Exponentiated to the length $\pi /2$, we obtain a  broken $\mathcal F$-horizontal geodesic $\eta:[s,t] \to A_1$  that connects $x_1$ with $x_-$.   

 But, at any point $y\in A_1$, any pair of $\mathcal F$-horizontal vectors $h,v \in  T_y M$ are tangent to a section of 
$\mathcal  F$, whenever $v$ is tangent to $A_1$ and $h$ orthogonal to $A_1$ (i.e., $h$ is $\mathcal F'$-horizontal).
Therefore, if $d(x_2, \eta (r)) =\pi /2$, for some $r\in (s,t]$, then  $x_2, \eta (r)$ and $\eta (r-\epsilon )$ are contained in some section of $\mathcal F$.  Thus $d(x_2,\eta (r-\epsilon ))=\pi /2$ as well. Running $\eta$ backwards from $x_-$ to
$x_1$ we deduce $d(x_1,x_2)=\pi /2$. 
\end{proof}

\subsection{The dual foliation}
We are going to prove \lref{postpone} now.

First, let us assume that $\Delta$ is irreducible.  
Take a dual  leaf $X' =L^\# (x)$ of $\mathcal F$. We have seen in \cref{mainirred}  that $X'$ has diameter at most $\pi$
with respect to $d^{hor}$. Since $X$ consists of  simplices of the same size, any point of the dual leaf $L^\# (x)$ can be connected with $x$ by a broken horizontal geodesic with at most $n$ breaks (for some $n$ depending only on $\Delta$), of total length at most $\pi$. Since a limit of a sequence of such broken horizontal geodesic is again a broken horizontal geodesic, we deduce that $L^\#$ is compact. 

If the quotient $\Delta$ is reducible as spherical join, the conclusion follows in the same way, using   \pref{mainred}.

\subsection{More conclusions} 
After having closed the gap  \lref{postpone}   the proof of \pref{decompose} is complete.  We deduce:

\begin{cor}
If the foliation $\mathcal F$ is indecomposable, there is  only one dual leaf.
\end{cor}

{\bf From now on, in addition to our assumptions from section \ref{5},  $\mathcal F$ will be  indecomposable as a product}.

\section{Polars and meridians} \label{7}
We assume here that $\Delta$ is reducible as a spherical join $\Delta = \Delta _1 \ast \Delta _2$, and are going to prove  that $M$ has rank $1$.  

Set $A_i = p^{-1} (\Delta _i)$. Using \pref{mainred} we know that $A_i$ are smooth manifolds, and  we have 
$d(x_1,x_2)=\pi /2$, for all $x_i \in A_i$.
Since there is only one dual leaf of $\mathcal F$, any point $x$ in $M$ lies on a unique shortest geodesic from $A_1$ to $A_2$.  Finally, any geodesic that
starts horizontally on $A_1$ is closed of period $\pi $ or of period $2\pi$.

We are going to use a few easy facts about  \emph{polars} and \emph{meridians} (\cite{Nagano}, \cite{Chen}). 
Recall that in a symmetric space $M$, a \emph{polar} of a point $o$ is a connected component of the fixed point set
of the geodesic symmetry $s_o$ at the point $o$. The \emph{meridian}  $M^- (p)$ through a point $p$ in a polar $M^+$
of $o$ is the component through $p$ of the fixed point set of the isometry $s_o \circ s_p$.   
We recall that the tangent spaces at $p$  of $M^+$ and $M^-$ are complementary orthogonal subspaces.
Moreover, the rank of the meridian $M^-$ is equal to the rank of $M$.

Assume  first that all sections of $\mathcal F$ are projective spaces, i.e., all horizontal geodesics are closed of period $\pi$. Then, for any $o\in A_1$, the reflection $s_o$ at $o$ must leave
$A_2$ pointwise fixed. Choose any point $p\in A_2$. Then $p$ is contained in a polar $M^+$ of $o$, that contains $A_2$.
Thus the tangent space to the meridian $M^- (p)$ through $p$ is contained in the orthogonal space to $A_2$.  Thus,
in the symmetric space $M^{-} (p)$, all geodesics starting at $p$ are closed. Hence $M^{-} (p)$ has rank one.
 Therefore  the  rank of $M$ must be $1$  as well.

Assume now that all sections are spheres. For any $o\in A_1$, the geodesic symmetry $s_o$ leaves $A_2$ 
invariant, but no point in $A_2$ is fixed by $s_o$. Therefore, $s_o (A_1)= A_1$ as well.
Moreover, all polars of $o$ must be contained in $A_1$.  Let $p\neq o$ be fixed by $s_o$ and let  $M^- (p)$ be the meridian through $p$. Since the polar $M^+ (p)$  is  contained in $A_1$, the normal space to $A_1$ is tangent to the meridian. Hence, the meridian contains  $A_2$.   In the meridian $M^- :=M^-(p)$, the point $p$ is a \emph{pole} of $o$, meaning  that in the symmetric space $M^{-}$, the point $p$ is a one-point  polar of $o$. In such a case  there is a two-to-one covering 
$c: M^- \to M_1$, such that  $M_1$ is symmetric and $c$ sends $o$ and $p$ to the same point $\bar o$ (\cite{Nagano}). 
 In $M_1$, the projections of horizontal geodesics
starting in $\bar o$ have period $\pi$.  Hence a  polar of $\bar o$ inside  $M_1$ contains the image 
$\bar A_2$ of $A_2$.  Therefore, the meridian  in $M_1$  through any point of $\bar A_2$ is orthogonal to $\bar A_2$.
Thus all geodesics in this meridian are closed and it must have rank 1. Due to \cite{Nagano}, $M_1$ and hence $M$ must have 
rank 1 as well.

\section{Topological  buildings}  \label{8}
We assume here that $\Delta  =\mathbb S^k /\Gamma$ is an  irreducible Coxeter simplex. Moreover, we assume that the universal covering $\tilde X$ is a building.
Due to \cref{mainirred}, the last assumption is always fulfilled if $k\geq 3$.   Again we are going to prove that $M$ has rank $1$.

We denote by $K$ the fundamental group of $X$ acting on $\tilde X$ by deck transformations. By 
$\pi: \tilde X \to X$   we denote the  projection. We are going to define a compact $K$-invariant topology $\mathcal T$ on 
$\tilde X$.

 In order to do so, we will use the following construction several times. Let $N$ be a compact, geodesic, simply connected metric space (for us an interval or a disc). Let $h_i:N\to \tilde X$ be a sequence of uniformly Lipschitz maps. 
Consider the projections $\bar h_i = \pi \circ h_i :N\to \tilde X \to X \to M$.  Then as Lipschitz maps to the compact manifold $M$ the sequence is equicontinuous and we find a subsequence converging to a Lipschitz map $\bar h :N\to M$.
 Since all $\bar h_i$ map a small  neighborhood of any point  into the union of sections through the image, the same is true 
for  $\bar h$. Thus $\bar h$ is in fact a Lipschitz map to $X$. Assuming that all $h_i$ send a base point of 
$N$ to the same point $q\in \tilde X$ we  have a unique lift of $\bar h$ to a Lipschitz map $h:N\to \tilde X$ sending the base point 
of $N$ to $q$.  We will say that the sequence $h_i$ \emph{weakly subconverges} to $h$.

To define the topology $\mathcal T$, we first fix a point $q\in \tilde X$.  We will say that a point $p \in \tilde X $ is contained in the  $\mathcal T$-closure of  a subset $C\subset \tilde X$ if and only if for some  sequence $p_n\in C$ there is a curve $\gamma$ from $q$ to $p$ and  a sequence of shortest geodesics $\gamma _n$ from $q$ to $p_n$, such that
$\gamma _n$ weakly converges to $\gamma$.

 From the Theorem of Arzela-Ascoli  we see that the topology $\mathcal T$ is  sequentially compact. Moreover,  $(\tilde X, \mathcal T)$ has a dense 
countable subset.
We are going to prove that $\mathcal T$ is Hausdorff and does not depend on the base point $q$.

\begin{lem}
Let $p_n$ converge to $p$ in  $\mathcal T$. Let $\gamma ' _n :[0,1] \to \tilde X$ be a curve of length $\leq L <\infty$
from $q$ to $p_n$, parametrized proportionally to arc length. Let $\gamma ' :[0,1 ]\to \tilde X $ be a weak limit of
 $\gamma ' _n$.   Then $\gamma '$ ends  in $p$.
\end{lem}

\begin{proof}
 Let $\gamma _n, \gamma $ be as in the definition of $\mathcal T$ above. Let $r_n:S^1 \to \tilde X$ be the concatenation
of $\gamma _n$ and of the reversed $\gamma_n '$.  Since $\tilde X$ is a spherical building of dimension at least $2$,
$r_n$  can be retracted to a point uniformly, i.e.,  $r_n$ can be extended to some   $L'=L'(L)$-Lipschitz map 
$r_n : D^2 \to \tilde X$  (Straighten $\gamma  ' _n$ to be a broken geodesic with uniformly many geodesic parts, using that
the injectivetiy radius is $\pi$.  Then subdivide $S^1$ into  uniformly finitely many intervals, such that  $q$ and the image of any of  these intervals are contained together in an apartment).  Consider now a weak limit $r:D^2\to \tilde X$ of the sequence $r_n$.
By construction, the left half-circle in $r(S^1)$ is $\gamma$ and the right half-circle in $r(S^1)$  is $\gamma '$.
 \end{proof}

The lemma implies that a sequence cannot converge in $\mathcal T$  to two different points.
Thus $\mathcal T$ is Hausdorff.  Since $\mathcal T$  is separable and sequentially compact it is a compact metrizable topology.
Taking another point $q' \in \tilde X$ and considering concatenations with a fixed geodesic from $q$ to $q'$,
the lemma implies that the topology does not depend on the base point $q$. Thus it is defined only in terms of the projection
$\pi:\tilde X \to M$.  Therefore, it is invariant under the action of $K$.

By construction, a small  metric ball  around any point $x \in \tilde X$ is sent by $\pi$ bijectively  onto a small ball 
in (the exponential image of) the normal space to the leaf through $\pi (x)$. The topology $\mathcal T$ we have defined, restricts to this ball as  the usual Euclidean topology in the normal space.  Thus the intersections of a preimage of a regular leaf and a  small ball  around any point of $\tilde X$  is connected.

Thus we have a compact irreducible building $(\tilde X ,\mathcal T)$, in the sense of \cite{burns}. Since the   preimages of the leaves of $\mathcal F$ (i.e., points of the same type in $\tilde X$, in other words, the set of chambers of the building) contain  non-trivial  connected subsets, the set of chambers    is connected
(\cite{Kramer}).  From \cite{burns} it follows, that  the space $(\tilde X, \mathcal T )$ must be homeomorphic to a sphere.
 Moreover, the building 
is the spherical building of a simple non-compact real Lie group and can be identified with the boundary at infinity of
a non-compact irreducible symmetric space. In particular,  the group of automorphisms of the building acts on the sphere in a linear way.

Consider now the action of $K$ on $\tilde X$.  The orbit of any point is the preimage of a point under the continuous projection $\pi :\tilde X \to M$.  Thus it is a compact set.  The group of topological automorphisms $G$ of the compact building 
$\tilde X$ is locally compact with respect to the compact-open topology (\cite{burns}).  We claim that $K$ is a compact subgroup of $G$.   In fact, take a sequence $g_n \in K$.  Choose a point $p \in \tilde X$. Then there is some $g \in K$,
such that $g_n \cdot p$ converges to $g \cdot p$. We call $h_n =g ^{-1} g_n$ and have $h_n \cdot  p \to p$.
We claim that $h_n$ converges to the identity in $G$.

Choose a shortest geodesic $\gamma _n$ from $p_n = h_n (p)$ to $p$.  Choose now a point $q \in \tilde X$ and a shortest geodesic $\eta$ from $p$ to $q$.  Then $h_n (q)$ is given by the lift starting at $p$ of the projection of the concatenation of
$\gamma _n$ and $\eta$.    These projections converge to a curve which lifts to a curve ending at $q$.  Hence, $q_n$ converges to $q$.   Therefore, $g_n$ converges to $g$. Thus  $K$ is compact.

Thus our compact group $K$ of automorphisms acts freely and linearly  on the sphere $\tilde X$. The projection map
$\pi :\tilde X \to M$ has as fibers the orbits of $K$, hence $M$ is the quotient space  $M=\tilde X/K$.
By assumption, $M$ is simply connected, hence $K$ is connected. The only connected groups that act without fixed points
on a sphere are the trivial group,  $\mathrm U (1)$ and $\mathrm {SU} (2)$. Then the quotient space $M$ is homeomorphic to a sphere,  or  projective space over the complex
or over the quaternions. 

But  only symmetric spaces of rank 1 have the topology of a sphere or of a projective space (cf., \cite{Ziller}).

\section{Polar foliations on symmetric spaces of rank one}  \label{9}
Polar {\it actions} on symmetric spaces of rank one have been studied and classified in \cite{Pth1}. 
The geometric structure of polar foliations on such spaces is not more complicated. The following result is  folklore (cf. \cite{Dominguez}):

\begin{prop}
Let  $M$ be a projective space $F\mathbb P^m$, where $F$ denotes the field of complex or quaternionic numbers.
Let $h: \mathbb S^n  \to M$ be the Hopf fibration from the round sphere. If $\mathcal F$ is a polar foliation
on $M$ then its lift $\hat {\mathcal F} := h^{-1} (\mathcal F )$ is a polar foliation on $\mathbb S^n$. 
\end{prop}

\begin{proof}
We normalize our space, such that the round sphere $\mathbb S^n$ has curvature $1$.
The Hopf fibration $h$ is a Riemannian submersion, hence $\hat {\mathcal F}$ is a singular Riemannian foliation on $\mathbb S^n$,
with the same quotient space $\Delta = M/\mathcal F = \mathbb S^n / \hat {\mathcal F}$. If the dimension $k$ of $\Delta $ is $1$,
then $\hat {\mathcal F}$ is of codimension $1$, hence polar. If $k\geq 2$, then $\hat {\mathcal F}$ is polar if and only if the orbifold $\Delta$ has constant 
curvature $1$.  Thus $\hat {\mathcal F}$ is polar if and only if the sections of $\mathcal F$ have constant curvature $1$.

 The sections of $\mathcal F$ are either spheres or projective spaces of constant curvature.  
 Maximal totally geodesic spheres in $M$
are given by  the  projective lines  $F\mathbb P^1$
 (\cite{Pth1}) .  However, any pair of such projective lines intersect in at most one point and never in a 1-dimensional or $3$-dimensional subset.  Thus if all sections are projective lines, the foliation must be regular. This contradicts \lref{easyprod}.

Hence all sections of $\mathcal F$ are real projective spaces $\mathbb {R P}^k$.
 But such projective spaces have curvature $1$ (cf. \cite{Pth1}). 
\end{proof}

The same proof as above shows:

\begin{prop}
Let $\mathcal F$ be a polar foliations on  the Cayley projective plane.  Then either  $\mathcal F$  has codimension  $1$
or the sections of $\mathcal F$ are real projective subspaces $\mathbb {R P}^2$ and $\mathcal F$ has codimension $2$. 
\end{prop}

Combining the above propositions with the results of Section \ref{7} and  Section \ref{8}, we finish the proof
of \tref{mainthm} and \tref{irredthm}.

\bigskip

\noindent\textbf{Acknowledgements} For helpful discussions and useful comments on preliminary versions of the paper, I am  very grateful to Karsten Grove, Marco Radeschi, Gudlaugur Thorbergsson  and  Stephan Wiesendorf.

\bibliographystyle{alpha}
\bibliography{symmetric}

\end{document}